\documentclass[11pt]{amsart}

\usepackage{amssymb, mathtools, delimset}
\usepackage{microtype}
\usepackage{enumitem, booktabs, array}
\setlist[enumerate]{
labelindent=0pt,
leftmargin=*,
align=left,
itemsep=.35\baselineskip,
topsep=.15\baselineskip,
partopsep=0pt,
parsep=0pt
}
\usepackage{lmodern}
\usepackage[T1]{fontenc}
\usepackage[utf8]{inputenc}
\usepackage[numbers, sort&compress, longnamesfirst]{natbib}
\usepackage{etoolbox}
\usepackage{hyphenat}
\makeatletter
\patchcmd{\NAT@citexnum}
{\@citea \NAT@test{\@ne}\NAT@spacechar\NAT@mbox{\NAT@super@kern\NAT@@open}}
{\@citea \nohyphens{\NAT@test{\@ne}}\nobreakspace{}\NAT@mbox{\NAT@super@kern\NAT@@open}}
{}
{\PackageError{nonbreaking-citet}
{Failed to patch \string\citet}
{The internal definition of natbib may have changed.}}
\makeatother

\usepackage{xcolor}
\definecolor{BIT}{cmyk}{1, 0, 1, 0}
\usepackage[
pagebackref,
colorlinks=true,
linkcolor=blue,
citecolor=BIT,
urlcolor=blue
]{hyperref}

\usepackage[capitalize, noabbrev, nameinlink, sort&compress]{cleveref}

\makeatletter
\let\cref@original@gettype\cref@gettype
\renewcommand{\cref@gettype}[2]{%
\ifstrequal{#1}{lem:AlphaOmega}
{\def#2{alphaomega}}
{\cref@original@gettype{#1}{#2}}%
}
\makeatother
\crefformat{alphaomega}{#2the Alpha--Omega lemma#3}
\Crefformat{alphaomega}{#2The Alpha--Omega lemma#3}

\crefname{def}{Def.}{Defs.}
\Crefname{def}{Def.}{Defs.}
\creflabelformat{def}{~\upshape(#2#1#3)}
\crefname{section}{\S}{\S\S}
\Crefname{section}{Section}{Sections}
\AddToHook{begindocument/end}{%
\crefformat{section}{#2\S#1#3}
\Crefformat{section}{#2Section~#1#3}
}

\makeatletter
\newcommand\creflabel[2][\@currentcounter]{%
\crefalias{\@currentcounter}{#1}\label{#2}}
\makeatother

\crefname{table}{Tab.}{Tabs.}
\Crefname{table}{Tab.}{Tabs.}
\crefname{equation}{Eq.}{Eqs.}
\Crefname{equation}{Eq.}{Eqs.}
\crefformat{equation}{Eq.~(#2#1#3)}

\hypersetup{
pdftitle={An e-positive classification for complete multipartite graphs},
pdfauthor={Ariel Y. Sun, David G.L. Wang, and Watson Z.Y. Wang},
pdfkeywords={chromatic symmetric function, e-positivity,
complete multipartite graph, acyclic orientation, Schur positivity}
}

\numberwithin{equation}{section}
\numberwithin{table}{section}
\newtheorem{theorem}{Theorem}[section]
\newtheorem{lemma}[theorem]{Lemma}
\newtheorem{proposition}[theorem]{Proposition}
\newtheorem{corollary}[theorem]{Corollary}
\makeatletter
\let\cref@original@resetby\cref@resetby
\renewcommand{\cref@resetby}[2]{%
\@ifundefined{alias@ctr@#1}
{\cref@original@resetby{#1}{#2}}
{\edef\cref@aliascounter{\csname alias@ctr@#1\endcsname}%
\expandafter\cref@original@resetby\expandafter{\cref@aliascounter}{#2}}%
}
\makeatother

\newcommand{\coeff}[2]{\brk[s]!{#1}#2}

\allowdisplaybreaks

\title[$e$-Positivity for complete multipartite graphs]{
An $e$-positive classification\\[5pt]
for complete multipartite graphs}
\author{Ariel Y. Sun}
\address{School of Mathematics and Statistics, Beijing Institute of Technology,
Beijing 102400, China}
\email{ariel@bit.edu.cn}

\author{David G.L. Wang\textsuperscript{*}}
\address{School of Mathematics and Statistics, Beijing Institute of Technology,
Beijing 102400, China}
\email{glw@bit.edu.cn}
\thanks{\textsuperscript{*}Corresponding author.}
\thanks{The work of David Wang was supported by the General Program of the
National Natural Science Foundation of China (Grant No.~12171034).}

\author{Watson Z.Y. Wang}
\address{School of Mathematics and Statistics, Beijing Institute of Technology,
Beijing 102400, China}
\email{watsonzywang@bit.edu.cn}

\date{}

\begin{document}

\begin{abstract}
Shelburne and van Willigenburg characterize the Schur-positive complete multipartite graphs and leave open the $e$-positivity of the graph family~$K_{32^\beta}$. We resolve this question by giving an explicit, manifestly nonnegative~$e$-expansion of its chromatic symmetric function, with coefficients expressed in terms of restricted-injection numbers. The key step in our proof is a linear-functional reduction of multivariate marker extraction. Starting from the elementary--monomial Cauchy identity, we derive a general coefficient-extraction formula for complete multipartite graphs. For the family considered here, only the constant and linear terms in the third elementary symmetric polynomial contribute. By combining a first-order expansion in one auxiliary root with the factorial weights from marker extraction, we obtain a single linear-functional evaluation of both contributions. The functional's evaluations on odd and even Dickson polynomials express the individual coefficients through restricted-injection numbers and their successive differences, making nonnegativity explicit. Together with their classification, this characterizes all~$e$-positive complete multipartite graphs and shows that every Schur-positive complete multipartite graph is~$e$-positive.
\end{abstract}

\maketitle

\section{Introduction}

In 1995,
\citet{Sta95} introduced the chromatic symmetric function~$X_G$,
a symmetric-function refinement of the chromatic polynomial that records
substantially more information about a graph than its chromatic polynomial alone.
Positivity questions for~$X_G$ quickly became central,
partly because of their connections with algebraic combinatorics and
representation theory.
In particular,
Stanley reformulated the Stanley--Stembridge conjecture,
introduced by \citet{SS93},
as the assertion that the chromatic symmetric function of the incomparability
graph of every $(3+1)$-free poset is~$e$-positive.
\citet{Hik24X} proved the conjecture.
The present paper does not rely on this work.

Beyond incomparability graphs of $(3+1)$-free posets,
the problem of characterizing graphs with~$e$-positive
chromatic symmetric functions remains active.
\citet[Sec.~5]{Sta95} formulated the general question:
\[
\textit{``Which~$X_G$ is~$e$-positive?''}
\]
We focus here on positivity results for graph families that are not covered
in full by Hikita's theorem.
\citet{WW23-DAM} studied the~$e$- and Schur-positivity of spiders and broom trees.
\citet[Thm.~3.3]{QTW26} derived positive~$e$-expansions for
clique-path-cycle graphs.
\citet[Cor.~54 and 65]{TV26-vglue} proved~$e$-positivity for chains of
naturally labeled unit interval graphs and cycles,
formed by identifying the last vertex of each graph with the first vertex of
the next,
and for their cyclic closures obtained by identifying the two ends.
As particular cases,
these results recover the hat and cycle-chord positivity theorems that
\citet{WZ25} and \citet{Wang25},
respectively,
proved earlier.
\citet[Thm.~1.2]{CHW26} proved that clocks are~$e$-positive.
For further background,
see the introduction of \citet{GWZ25X} and the references therein.

The Stanley--Gasharov conjecture asserts that every claw-free graph is
Schur-positive.
\citet{Sta98} posed the conjecture and attributed it to Gasharov;
see also \citet{Gas99}.
It was disproved independently by \citet{MM26X} and \citet{Pra26X},
who found the same pair of counterexamples.
Beyond families already known to be~$e$-positive,
relatively few graph families have been proved Schur-positive.
\citet{Sv25} established Schur positivity for generalized net graphs,
and \citet{WW26Sabc} recently gave a complete classification of Schur positivity
for three-legged spiders.
For graph families that may contain claws,
\citet{WW20} characterized the Schur-positive complete bipartite and tripartite
graphs.
\citet[Thms.~4.7 and 4.18]{Sv26X} extended this classification to all
complete multipartite graphs.
For a partition~$\lambda$,
let~$K_\lambda$ denote the complete multipartite graph whose partite-set sizes
are the parts of~$\lambda$.

\begin{theorem}[\citeauthor{Sv26X}]
\label{thm:s.K}
Let $K_\lambda$ be a complete multipartite graph with~$\ell(\lambda)\ge 2$.
Then the following statements hold.
\begin{enumerate}[label=\textup{(\roman*)}]
\item If every part of~$\lambda$ belongs to $\{1,2\}$,
then $K_\lambda$ is~$e$-positive.
\item The graph~$K_\lambda$ is Schur-positive if and only if every part
of~$\lambda$ belongs to $\{1,2\}$ or~$\lambda=32^\beta$ for some $\beta\ge
1$.
\end{enumerate}
\end{theorem}

\citet{Sv26X} leave open the question of whether~$K_{32^\beta}$
is~$e$-positive for every integer~$\beta\ge 1$.
We answer this question affirmatively as follows.
For $d\ge 1$,
write $[d]=\{1,\dots,d\}$,
and set $[0]=\varnothing$.

\begin{theorem}
\label{thm:e.K}
For any integer $\beta\ge 1$,
the complete multipartite graph~$K_{32^\beta}$ is~$e$-positive as
\begin{align*}
X_{K_{32^\beta}}
=\beta!\brk4(&\beta(2\beta+3)a_{\beta+1}e_{2\beta+3}
+\beta\sum_{d=0}^{\beta-1}(2d+1)a_d
e_{(\beta+2+d)(\beta+1-d)} 
\\
&{}+\brk1{a_{\beta+1}+(2\beta^2+\beta-1)a_\beta}e_{(2\beta+2)1}
+e_{(\beta+1)^2 1} 
\\
&{}+\sum_{d=2}^{\beta}
\brk1{(4d-3)a_{d-1}+a_{d-2}}
e_{(\beta+d+1)(\beta+1-d)1}\brk4),
\end{align*}
where 
$a_d=\#\brk[c]1{f\colon[d]\hookrightarrow[2d]\,\colon\,
\text{$f(i)\ne i$ for every $i\in [d]$}}$.
Here the hook arrow indicates that~$f$ is injective.
\end{theorem}

The first few expansions are
\begin{align*}
X_{K_{32}}
&=35e_5+9e_{41}+e_{32}+e_{2^2 1},
\\
X_{K_{32^2}}
&=1988e_7+268e_{61}+12e_{52}+4e_{43}
+12e_{51^2}+2e_{3^2 1},
\\
X_{K_{32^3}}
&=162162e_9+14526e_{81}+630e_{72}+54e_{63}
+18e_{54}+384e_{71^2}+36e_{621}
+6e_{4^2 1}.
\end{align*}
The integers~$a_d$ count nonattacking rook placements on a deleted-diagonal board.
The first few values are $a_0=1$,
$a_1=1$,
$a_2=7$,
$a_3=71$,
$a_4=1001$,
$a_5=18089$.
See OEIS A002119 \cite{oeis2026} and \cref{sec:pre} for more details
on the numbers~$a_d$.

Combining \cref{thm:e.K,thm:s.K} yields a
characterization of~$e$-positive complete multipartite graphs.

\begin{corollary}
\label{cor:es.K}
Let $K_\lambda$ be a complete multipartite graph.
Then the chromatic symmetric function~$X_{K_\lambda}$ is~$e$-positive if and
only if at least one of the following holds:
\begin{enumerate}[label=\textup{(\roman*)}]
\item 
the partition~$\lambda$ has length one;
\item 
every part of the partition~$\lambda$ belongs to $\{1,2\}$;
\item 
the partition has the form~$\lambda=32^\beta$ for some
integer~$\beta\ge 1$.
\end{enumerate}
As a consequence,
every Schur positive complete multipartite graph is~$e$-positive.
\end{corollary}

\begin{proof}
If $\lambda=n$,
then $K_\lambda$ is edgeless and $X_{K_\lambda}=e_1^n$ is~$e$-positive.
Suppose that $\ell(\lambda)\ge 2$.
Since~$e$-positivity implies Schur-positivity,
necessity follows from \cref{thm:s.K}(ii).
Conversely,
\cref{thm:s.K}(i) proves case~\textup{(ii)},
and \cref{thm:e.K} proves case~\textup{(iii)}.
\end{proof}

Complete multipartite chromatic symmetric functions have also been studied from
a basis-theoretic viewpoint.
\citet{CS21} treat the functions~$X_{K_\lambda}$ as a basis of symmetric
functions and give combinatorial interpretations for the transition coefficients
between this basis and the monomial basis,
as well as for coefficients of chromatic and Tutte symmetric functions in the
complete multipartite basis.

Our key idea is to combine marker extraction with a first-order expansion in
one auxiliary root.
Using the elementary--monomial Cauchy identity,
we convert the~$e$-coefficients into marker coefficients of monomial symmetric
polynomials (\cref{thm:e-coeff.K.general}).
For~$K_{32^\beta}$,
a first-order root expansion captures all surviving contributions and combines
them into a single linear functional evaluation (\cref{lem:FG.f}).
This avoids computing the full elementary-symmetric expansion.
Dickson-polynomial identities then express the coefficients in terms of
restricted-injection numbers and their successive differences
(\cref{lem:functional-evaluations}).

\Cref{sec:pre} collects the necessary background.
In \cref{sec:coefficient-extraction},
we establish the coefficient-extraction formulas.
In \cref{sec:evaluation},
we evaluate the coefficients and prove \cref{thm:e.K},
then discuss alternative orientation arguments.

\section{Preliminaries}\label{sec:pre}

A \emph{partition} of a positive integer~$n$ is a finite weakly decreasing
sequence~$\lambda$ of positive integers with sum~$n$,
denoted $\lambda\vdash n$.
Its number of parts is its \emph{length},
denoted $\ell(\lambda)$.
We write partitions compactly as $\lambda=\lambda_1\lambda_2\dotsm$.
An exponent records multiplicity,
so $32^\beta$ denotes $(3,2,\dots,2)$ with~$\beta$ parts equal to~$2$.
Parentheses delimit a part containing arithmetic operations or multiple digits;
for example,
$(2\beta+2)1$ has parts $2\beta+2$ and~$1$,
and $(12)1$ has parts $12$ and~$1$.
We use the same convention for indices of symmetric functions and graphs.
The transpose (or conjugate) partition~$\lambda'$ is defined by
$\lambda'_j=\lvert\{\,i\colon\lambda_i\ge j\,\}\rvert$
for every integer~$j\ge 1$;
in particular, $\lambda'_1=\ell(\lambda)$.
We also allow the unique partition of~$0$,
which has length~$0$.
Set $\lambda!=\prod_{i=1}^{\ell(\lambda)}\lambda_i!$.

We work in the ring~$\Lambda$ of symmetric functions over the rational
field~$\mathbb{Q}$,
where~$x=(x_1,x_2,\dots)$ is an infinite sequence of independent commuting
indeterminates.
The symbols~$e_\lambda(x)$,
$m_\lambda(x)$,
and $s_\lambda(x)$ denote the \emph{elementary},
\emph{monomial},
and \emph{Schur functions},
respectively.
We omit the variable~$x$ when no confusion can arise.
We identify every partition~$\lambda$ with any finite or infinite sequence
obtained by appending trailing zeros to~$\lambda$.
The same notation applies to a finite alphabet~$z=(z_1,\dots,z_r)$: setting all
variables after~$z_r$ to zero specializes each symmetric function to a symmetric
polynomial in~$z_1,\dots,z_r$.
For an alphabet~$y=(y_1,y_2,\dots)$ and an integer~$d\ge 0$,
let~$\Lambda^d(y)$ denote the vector space of homogeneous symmetric polynomials
of degree~$d$ in~$y$.
For a monomial~$M$ and a formal power series~$F$,
we write~$\coeff{M}{F}$ for the coefficient of~$M$ in~$F$.
A symmetric function is \emph{$e$-positive} (respectively,
\emph{Schur-positive}) if its coefficients in the elementary (respectively,
Schur) basis are nonnegative.

We also require the~$e$-$m$ form of the Cauchy identity,
see \citet[Eq.~(4.2')]{Mac95B}.
It will be used in the proof of \cref{thm:e-coeff.K.general}.

\begin{proposition}[Cauchy's identity]
\label{prop:Cauchy}
Let $x=(x_1,x_2,\dots)$ and $y=(y_1,y_2,\dots)$ be sequences of commuting
indeterminates,
respectively.
Then
\[
\prod_{i\ge 1}\prod_{j\ge 1}(1+x_i y_j)
=\sum_{\lambda}e_\lambda(x)m_\lambda(y),
\]
where the sum runs over all partitions~$\lambda$ of positive integers and the
unique partition of~$0$.
\end{proposition}

\emph{Vieta's formulas} provide the link between the roots and their elementary
symmetric functions;
see \citet[Eq.~(2.2)]{Mac95B} and the proof of \cref{thm:e-coeff.K.general}.

\begin{proposition}[Vieta's formulas]
\label{prop:Vieta}
Let~$R$ be a commutative ring,
let $f(z)\in R[z]$ be monic of degree~$n$,
and let $y=(y_1,\dots,y_n)$ be a sequence of roots of~$f(z)$ in a splitting
algebra over~$R$.
Then
\[
f(z)
=\prod_{a=1}^n(z-y_a)
=z^n-e_1(y)z^{n-1}+e_2(y)z^{n-2}{}+\dots+(-1)^ne_n(y).
\]
\end{proposition}

For a finite simple graph~$G=(V,E)$ and a set~$S$,
a \emph{proper coloring} with color set~$S$ is a map $\kappa\colon V\to S$ such
that adjacent vertices receive distinct colors.
Following \citet{Sta95},
we define the \emph{chromatic symmetric function} of~$G$ to be
the symmetric function
\[
X_G(x)
=\sum_{\kappa}\prod_{v\in V}x_{\kappa(v)},
\]
where the sum is over all proper colorings with color set~$\mathbb Z_{>0}$.
A graph~$G$ is \emph{$e$-positive} (respectively,
\emph{Schur-positive}) if its chromatic symmetric function~$X_G$ is
$e$-positive (respectively,
Schur-positive).

An \emph{independent set} in a graph~$G$ is a set of pairwise nonadjacent vertices,
and a \emph{clique} is a set of pairwise adjacent vertices.
The \emph{independence number}~$\alpha(G)$ is the maximum cardinality of an
independent set in~$G$.
More generally,
for every integer~$k\ge 1$,
let~$\alpha_k(G)$ be the maximum total cardinality of~$k$ pairwise disjoint
independent sets,
and let~$\omega_k(G)$ be the maximum total cardinality of~$k$ pairwise disjoint
cliques.
To narrow the range of partitions that can index nonzero~$e$-coefficients,
we will need the following Alpha--Omega Lemma of \citet[Lem.~2.10]{ST26}.

\begin{lemma}[Alpha--Omega Lemma]
\label{lem:AlphaOmega}
Let~$G$ be a graph on~$n$ vertices,
let $\mu\vdash n$,
and suppose that $\coeff{e_\mu}{X_G}\ne 0$.
For every integer~$k\ge 1$,
\[
\alpha_k(G)
\ge \mu'_1+\dots+\mu'_k
\qquad\text{and}\qquad
\omega_k(G)
\le \mu_1+\dots+\mu_k,
\]
where $\mu_1'\mu_2'\dotsm$ is the transpose of the partition $\mu=\mu_1 \mu_2
\dotsm$.
\end{lemma}

For two independent indeterminates~$r,s$,
the \emph{Dickson polynomials of the first kind} are characterized by
\[
D_k(r+s,\,rs)
=r^k+s^k
\qquad\text{for $k\ge 1$}.
\]
These polynomials go back to \citet{dickson1897}.
Their explicit formula is
\begin{equation}\label{eq:Dickson}
D_k(A,B)
=\sum_{q=0}^{\lfloor k/2\rfloor}
\frac{k}{k-q}\binom{k-q}{q}(-B)^q A^{k-2q}
\qquad\text{for $k\ge 1$};
\end{equation}
see \citet[Def.~9.6.1 and Thm.~9.6.3]{mullen-panario2013}.
For further background,
see \citet{lidl-mullen-turnwald1993}.
We use the convention 
\[
\binom nq=0
\qquad\text{if $q<0$ or $q>n$}.
\]
Define $\Delta_d=a_d-a_{d-1}$ for $d\ge 1$.

\begin{proposition}
\label{prop:a:Delta}
We have
\begin{align}
\label{eq:a}
a_d&=\frac{1}{d!}\sum_{q=0}^{d}(-1)^q\binom{d}{q}(2d-q)!
\quad\text{for $d\ge 0$},
\\
\label{eq:Delta}
\Delta_d
&=\frac{2d}{d!}
\sum_{q=0}^{d}
(-1)^q
\binom{d}{q}
(2d-q-1)!
\quad\text{for $d\ge 1$}.
\end{align}
Moreover,
$\Delta_1=0$,
and $\Delta_d=(4d-3)a_{d-1}+a_{d-2}>0$ for every integer~$d\ge 2$.
\end{proposition}

\begin{proof}
The first two identities are known.
The sum in \cref{eq:a} is equivalent to
\citet[Eq.~(1)]{chao-desjarlais-leonard2005},
while \cref{eq:Delta} corresponds to \citet[Eq.~(17)]{efimov2021}.
The identity $\Delta_1=0$ follows from $a_0=a_1=1$.
For $d\ge 2$,
\citet[Eq.~(16)]{efimov2021} gives the recurrence
\[
a_d
=2(2d-1)a_{d-1}+a_{d-2}.
\]
Subtracting~$a_{d-1}$ from both sides gives the final formula.
It is positive since the numbers~$a_k$ count injections and are positive.
\end{proof}

\section{Coefficient extraction and reduction to a linear functional}
\label{sec:coefficient-extraction}

We develop the coefficient-extraction formulas used in \cref{sec:evaluation}.
From this point onward,
let $\beta\ge 1$ and set $\lambda=32^\beta$ and
$K_\lambda=K_{32^\beta}$.
We first determine which coefficients can be nonzero.

\begin{lemma}
\label{lem:restrction.e-support}
Let $\lambda=32^\beta$ and $\mu\vdash 2\beta+3$.
If $\coeff{e_\mu}{X_{K_\lambda}}\ne 0$,
then $\ell(\mu)\le 3$,
where the equality holds only if $\mu_3=1$.
\end{lemma}

\begin{proof}
Suppose that $\coeff{e_\mu}{X_{K_\lambda}}\ne 0$.
\Cref{lem:AlphaOmega} with $k=1$ gives
\[
\ell(\mu)
=\mu'_1
\le \alpha_1(K_\lambda)
=\alpha(K_\lambda)
=3.
\]
On the other hand,
any two pairwise disjoint cliques contain at most two vertices from each of the
$\beta+1$ partite sets of~$K_\lambda$,
and this bound is attained by choosing one vertex from each partite set for each
clique.
Thus \cref{lem:AlphaOmega} with $k=2$ gives
\[
2\beta+2
=\omega_2(K_\lambda)
\le \mu_1+\mu_2
=2\beta+3-\mu_3.
\]
Thus $\mu_3\le 1$.
If $\ell(\mu)=3$,
then~$\mu_3$ is a positive part,
and hence $\mu_3=1$.
\end{proof}

\Cref{lem:restrction.e-support} restricts our calculation to a small family of
partitions for proving \cref{thm:e.K}.
Next,
we derive a marker-variable formula valid for every complete multipartite graph.

\begin{theorem}
\label{thm:e-coeff.K.general}
Let $\eta=\eta_1\eta_2\dotsm\vdash n$.
Assign a marker variable~$u_i$ to the~$i$th partite set of~$K_\eta$,
set $u=(u_1,u_2,\dots)$,
and write $u^\eta=u_1^{\eta_1} u_2^{\eta_2} \dotsm$ for the \emph{marker monomial}.
Then
\[
X_{K_\eta}
=\eta!
\sum_{\mu\vdash n}
\coeff{u^\eta}{m_\mu(y)}\,e_\mu(x),
\]
where $y=(y_1,\dots,y_{\eta_1})$ is the sequence of roots of the polynomial
\[
z^{\eta_1}
-E_1 z^{\eta_1-1}
+E_2 z^{\eta_1-2}
-\dots
+(-1)^{\eta_1}E_{\eta_1},
\]
whose coefficients are given by
\[
E_j
=\frac{1}{j!}
\sum_{\substack{1\le i\le \ell(\eta),\ \eta_i\ge j}}
u_i^j,
\qquad\text{for each integer $1\le j\le \eta_1$.}
\]
\end{theorem}

The coefficient extraction in \cref{thm:e-coeff.K.general} is well defined:
expand the monomial symmetric polynomial~$m_\mu(y)$ in the
elementary symmetric polynomials~$e_i(y)$,
which equal~$E_i$ by \cref{prop:Vieta}.
Thus
\[
m_\mu(y)
\in \mathbb{Q}[E_1,\dots,E_{\eta_1}]
\subseteq\mathbb{Q}[u_1,\dots,u_{\ell(\eta)}].
\]

\begin{proof}
Write~$V_i$ for the partite set marked by~$u_i$,
so $\lvert V_i \rvert=\eta_i$.
For a fixed color~$c$,
the proper-coloring condition permits~$c$ either to be unused or to occur~$j$
times in exactly one partite set.
Multiplying over all colors,
extracting the marker monomial that records all vertices,
and multiplying by the scalar factor $\eta!$ to restore the labelings within the
partite sets gives
\[
X_{K_\eta}
=\eta!
\,\coeff{u^\eta}{
\prod_{c\ge 1}
\left(1+\sum_{j=1}^{\eta_1}E_j x_c^j\right)}.
\]
By \cref{prop:Vieta,prop:Cauchy},
\[
\prod_{c\ge 1}
\left(1+\sum_{j=1}^{\eta_1}E_j x_c^j\right)
=\prod_{c\ge 1}\prod_{a=1}^{\eta_1}(1+x_c y_a)
=\sum_\mu e_\mu(x)m_\mu(y).
\]
Extracting the marker monomial forces degree~$n$,
and only partitions~$\mu\vdash n$ remain,
proving the desired identity.
\end{proof}

For the coefficient extraction in \cref{thm:e-coeff.K.general},
the computation may be carried out in the quotient ring
\[
\mathbb{Q}[u_1,\,\dots,\,u_{\ell(\eta)}]
\Big/
\brk2{u_1^{\eta_1+1},\,\dots,\,u_{\ell(\eta)}^{\eta_{\ell(\eta)}+1}},
\]
since every monomial containing a power of~$u_i$ greater than~$\eta_i$ may be
ignored.
This observation is essential for simplifying the calculations below.

We now apply \cref{thm:e-coeff.K.general} to the specific graph~$K_\lambda$.

\begin{corollary}
\label{cor:e-coeff.K.special}
For every partition $\mu\vdash 2\beta+3$,
\[
\coeff{e_\mu}{X_{K_\lambda}}
=\lambda!\coeff{M}{m_\mu(y)},
\]
where the sequence~$y=(y_1,\,y_2,\,y_3)$ is defined by its
elementary symmetric polynomials
\[
e_1(y)
=u+\sum_{i=1}^{\beta}v_i,\qquad
e_2(y)
=\frac{u^2+\sum_{i=1}^{\beta}v_i^2}{2},\qquad\text{and}\qquad
e_3(y)
=\frac{u^3}{6},
\]
and $M=u^3v_1^2v_2^2\dotsm v_\beta^2$.
\end{corollary}

\begin{proof}
Immediate by applying \cref{thm:e-coeff.K.general} with~$\eta=\lambda$.
\end{proof}

Let $E_i=e_i(y)$ for $i=1,2,3$.
For surviving coefficients in the extraction,
every term divisible by~$E_3^2=u^6/36$ has~$u$-degree at least~$6$,
whereas the marker monomial~$M$ has~$u$-degree~$3$.
Thus subsequent calculations may be performed modulo~$E_3^2$.
The monomials that remain have the form $E_2^sE_1^t$ or $E_2^sE_1^tE_3$.
The following lemma evaluates the marker coefficients of both types.

\begin{lemma}
\label{lem:marker-extraction}
Let $s,t\ge 0$ and $r\in \{0,1\}$ such that $3r+2s+t=2\beta+3$.
Then
\[
\lambda!\coeff{M}{E_3^rE_2^sE_1^t}
=s!t!\brk4{\binom{\beta}{s}+3\binom{\beta}{s-1}\mathbf 1(r=0)},
\]
where $\mathbf 1$ is the indicator function.
\end{lemma}

\begin{proof}
Suppose first that $r=1$.
The factor $E_3=u^3/6$ supplies all three copies of~$u$ in the marker
monomial~$M$.
Choose the~$s$ variables whose squared terms~$v_i^2/2$ are supplied by~$E_2^s$;
the factor~$E_1^t$ then supplies two copies of each of the remaining $\beta-s$
variables.
Since $2s+t=2\beta$,
multiplying the resulting coefficient by~$\lambda!$ gives
\[
\lambda!
\,\coeff{M}{E_3 E_2^s E_1^t}
=\lambda!
\binom{\beta}{s}
\frac{s!}{2^s}
\frac{t!}{3!\,2^{\beta-s}}
=s!t!\binom{\beta}{s}.
\]

Now suppose that $r=0$.
There are exactly two ways to supply the factor~$u^3$ in~$M$.
First,
the factor~$E_1^t$ may supply all three copies of~$u$.
In this case,
choose the~$s$ variables whose squared terms~$v_i^2/2$ are supplied by~$E_2^s$;
the factor~$E_1^t$ also supplies two copies of each of the remaining $\beta-s$
variables.
This contribution,
after multiplication by~$\lambda!$,
is
\[
\lambda!
\binom{\beta}{s}
\frac{s!}{2^s}
\frac{t!}{3!\,2^{\beta-s}}
=s!t!\binom{\beta}{s}.
\]
Second,
the factor~$E_2^s$ may supply $u^2/2$,
while~$E_1^t$ supplies the remaining copy of~$u$.
The other $s-1$ selections from~$E_2^s$ supply squared terms~$v_i^2/2$,
and~$E_1^t$ supplies two copies of each of the remaining variables.
This contribution is
\[
\lambda!
\binom{\beta}{s-1}
\frac{s!}{2^s}
\frac{t!}{2^{\beta-s+1}}
=3s!t!\binom{\beta}{s-1}.
\]
Summing the two contributions yields the stated formula.
\end{proof}

The factorial weights in \cref{lem:marker-extraction} motivate the following
linear functional.
For an integer~$h\ge 0$,
define a linear functional on polynomials in~$A,B$ of~$B$-degree at most~$h$ by
\[
\mathcal L_h(A^tB^s)
=\frac{t!}{(h-s)!},
\qquad\text{if $t\ge 0$ and $0\le s\le h$}.
\]

\begin{lemma}[Linear functional reduction]
\label{lem:FG.f}
Let $\beta\ge 1$ and $f\in \Lambda^{2\beta+3}(y)$,
where $y=(y_1,y_2,y_3)$.
Define $F,G\in \mathbb Q[A,B]$ by
\[
f(r,s,w)
=F(r+s,rs)+wG(r+s,rs)+O(w^2).
\]
Here $O(w^2)$ denotes a polynomial divisible by~$w^2$.
Then $\lambda!\coeff{M}{f(y)}=\beta!\mathcal L_{\beta+1}(\beta F+G)$.
\end{lemma}

\begin{proof}
Let $n=2\beta+3$.
By the fundamental theorem of symmetric polynomials,
we have $f(y)=H(E_1,E_2,E_3)$ for some polynomial~$H$ of weighted degree~$n$
when $E_1,E_2,E_3$ have weights $1,2,3$.
Expand~$H$ in powers of~$E_3$,
we obtain
\[
f(y)
=H_0(E_1,E_2)+E_3 Q(E_1,E_2)+O(E_3^2),
\]
where $H_0\in \Lambda^n$ and $Q\in \Lambda^{n-3}$.
Setting $y=(r,s,0)$ gives $H_0=F$.
Set $y=(r,s,w)$ with $A=r+s$ and $B=rs$.
Then $E_1=A+w$,
$E_2=B+Aw$,
and $E_3=Bw$.
Consequently,
\[
f(r,s,w)
=F(A+w,B+Aw)+BwQ(A+w,B+Aw)+O(w^2).
\]
The discarded terms contain $E_3^2=B^2w^2$,
so they do not affect the coefficient of~$w$.
First-order polynomial expansion gives
\begin{align*}
F(A+w,B+Aw)
&=F(A,B)+wF_A(A,B)+AwF_B(A,B)+O(w^2),
\quad\text{and}
\\
BwQ(A+w,B+Aw)
&=BwQ(A,B)+O(w^2).
\end{align*}
Here the subscripts denote partial derivatives.
Comparison with the defining expansion yields
\begin{equation}\label{pf:first-order-linear-term}
G=F_A+AF_B+BQ.
\end{equation}

Set $h=\beta+1$.
For a monomial~$A^tB^s$ of~$F$,
we have $t+2s=2\beta+3=2h+1$ and $0\le s\le h$.
By \cref{lem:marker-extraction},
\[
\lambda!\coeff{M}{E_1^tE_2^s}
=s!t!\left(\binom{\beta}{s}+3\binom{\beta}{s-1}\right)
=\beta!t!\left(\frac{h-s}{(h-s)!}
+\frac{3s}{(h-s)!}\right)
=\frac{\beta!t!(h+2s)}{(h-s)!}.
\]
On the other hand,
since $t+1=2(h-s+1)$,
differentiation gives
\begin{align*}
\mathcal L_h \bigl((A^tB^s)_A+A(A^tB^s)_B \bigr)
&=\mathcal L_h \bigl(tA^{t-1}B^s+sA^{t+1}B^{s-1}\bigr)
\\
&=\frac{t(t-1)!}{(h-s)!}
+\frac{s(t+1)!}{(h-s+1)!}
=\frac{(1+2s)t!}{(h-s)!}.
\end{align*}
There is no issue with $(t-1)!$,
since $t=2h+1-2s\ge 1$.
It follows that
\[
\lambda!\coeff{M}{E_1^tE_2^s}
=\beta!\mathcal L_h \bigl(\beta A^tB^s
+(A^tB^s)_A+A(A^tB^s)_B \bigr).
\]
By linearity,
we have
$\lambda!\coeff{M}{F(E_1,E_2)}=\beta!\mathcal L_h(\beta F+F_A+AF_B)$.

For every monomial~$A^tB^s$ of~$Q$,
we have $t+2s=2\beta$ and $0\le s\le \beta$.
By \cref{lem:marker-extraction},
\[
\lambda!\coeff{M}{E_3 E_1^tE_2^s}
=s!t!\binom{\beta}{s}
=\frac{\beta!t!}{(\beta-s)!}.
\]
Since
\[
\mathcal L_h(A^tB^{s+1})
=\frac{t!}{(h-(s+1))!}
=\frac{t!}{(\beta-s)!},
\]
summing over the monomials of~$Q$ gives
$\lambda!\coeff{M}{E_3 Q(E_1,E_2)}
=\beta!\mathcal L_h(BQ)$.

Finally,
terms divisible by~$E_3^2$ have zero marker coefficient.
Adding the evaluations of~$F$ and~$E_3 Q$,
and using linearity and \cref{pf:first-order-linear-term},
we obtain
\[
\lambda!\coeff{M}{f(y)}
=\beta!\mathcal L_h(\beta F+F_A+AF_B)
+\beta!\mathcal L_h(BQ)
=\beta!\mathcal L_h(\beta F+G).
\qedhere
\]
\end{proof}

The factor shift and the odd and even Dickson-polynomial evaluations below
reduce all the coefficient families to~$a_d$ and~$\Delta_d$.

\begin{lemma}\label{lem:functional-evaluations}
For integers $0\le b\le h$ and every polynomial~$P$ of~$B$-degree at most~$h-b$,
\[
\mathcal L_h(B^bP)
=\mathcal L_{h-b}(P).
\]
As consequences,
$\mathcal L_d(D_{2d+1}(A,B))=(2d+1)a_d$ for $d\ge 0$,
and $\mathcal L_d(D_{2d}(A,B))=\Delta_d$ for $d\ge 1$.
\end{lemma}

\begin{proof}
Take $P=A^tB^s$,
where $t\ge 0$ and $0\le s\le h-b$.
Then
\[
\mathcal L_h(B^bA^tB^s)
=\frac{t!}{(h-(s+b))!}
=\frac{t!}{((h-b)-s)!}
=\mathcal L_{h-b}(A^tB^s).
\]
By linearity,
we obtain $\mathcal L_h(B^bP)=\mathcal L_{h-b}(P)$.
By \cref{eq:Dickson,eq:a},
\begin{align*}
\mathcal L_d(D_{2d+1})
&=\sum_{q=0}^{d}(-1)^q
\frac{2d+1}{2d+1-q}\cdot
\frac{(2d+1-q)!}{q!(2d+1-2q)!}\cdot
\frac{(2d+1-2q)!}{(d-q)!}
\\
&=(2d+1)\sum_{q=0}^{d}
\frac{(-1)^q(2d-q)!}{q!(d-q)!}
=(2d+1)a_d.
\end{align*}
By \cref{eq:Dickson,eq:Delta},
\[
\mathcal L_d(D_{2d})
=\sum_{q=0}^{d}(-1)^q
\frac{2d}{2d-q}
\frac{(2d-q)!}{q!(2d-2q)!}
\frac{(2d-2q)!}{(d-q)!}
=2d\sum_{q=0}^{d}
\frac{(-1)^q(2d-q-1)!}{q!(d-q)!}
=\Delta_d.
\qedhere
\]
\end{proof}

\section{Explicit \texorpdfstring{$e$}{e}-expansion and \texorpdfstring{$e$}{e}-positivity}
\label{sec:evaluation}

\Cref{lem:restrction.e-support} restricts the coefficient indices.
We compute the constant and first-order terms of~$m_\mu(r,s,w)$ and apply
\cref{lem:FG.f,lem:functional-evaluations}.

\begin{lemma}
\label{lem:coeffs.short}
For $\beta\ge 1$,
we have
\begin{align*}
\coeff{e_{2\beta+3}}{X_{K_\lambda}}
&=\beta!\,\beta(2\beta+3)a_{\beta+1},
\\
\coeff{e_{(2\beta+2)1}}{X_{K_\lambda}}
&=\beta!\brk1{a_{\beta+1}+(2\beta^2+\beta-1)a_\beta},
\quad\text{and}\\
\coeff{e_{(\beta+2+d)(\beta+1-d)}}{X_{K_\lambda}}
&=\beta!\,\beta(2d+1)a_d, 
\qquad\text{for $0\le d\le \beta-1$}.
\end{align*}
\end{lemma}

\begin{proof}
Put $n=2\beta+3$,
$A=r+s$,
and $B=rs$.
We suppress the arguments~$(A,B)$ for short.
For each index~$\mu$,
apply \cref{lem:FG.f} to~$f(y)=m_\mu(y)$:
\[
m_\mu(r,s,w)
=F+wG+O(w^2).
\]
We will derive the polynomials~$F$ and~$G$ for given partition~$\mu$ in
\cref{tab:FG.coeff.short}.

\begin{table}[htbp]
\caption{Constant and linear terms in the third root.}
\label{tab:FG.coeff.short}
\centering
\begin{tabular}{@{}>{$}l<{$}>{$}l<{$}>{$}l<{$}@{}}
\toprule
\text{Index }\mu & F & G
\\
\midrule
(2\beta+3) & D_{2\beta+3} & 0
\\
(2\beta+2)1 & BD_{2\beta+1} & D_{2\beta+2}
\\
(\beta+2+d)(\beta+1-d) & B^{\beta+1-d}D_{2d+1} & 0
\\
\bottomrule
\end{tabular}
\end{table}

For the top index $\mu=n=2\beta+3$,
we have $n\ge 5$ and $m_n(r,s,w)=r^n+s^n+w^n$.
Thus $F=r^n+s^n=D_n$ and $G=0$.
This proves the first row in \cref{tab:FG.coeff.short}.
For a two-part index~$ab$ with $a>b\ge 1$,
we have
\[
m_{ab}(r,s,w)
=(rs)^b(r^{a-b}+s^{a-b})
+w^b(r^a+s^a)+w^a(r^b+s^b).
\]
Therefore,
$F=(rs)^b(r^{a-b}+s^{a-b})=B^bD_{a-b}$ and $G=D_a \mathbf 1(b=1)$.
Taking special value pair $(a,b)$,
we obtain the second and third row.
This proves the entries in the last two columns.

Applying
\cref{cor:e-coeff.K.special,lem:FG.f,lem:functional-evaluations}
to these three rows gives
\begin{align*}
\coeff{e_n}{X_{K_\lambda}}/\beta!
&=\mathcal L_{\beta+1}(\beta D_{2\beta+3})
=\beta(2\beta+3)a_{\beta+1},
\\
\coeff{e_{(2\beta+2)1}}{X_{K_\lambda}}/\beta!
&=\mathcal L_\beta(\beta D_{2\beta+1})
+\mathcal L_{\beta+1}(D_{2\beta+2})
\\
&=\beta(2\beta+1)a_\beta+\Delta_{\beta+1}
=a_{\beta+1}+(2\beta^2+\beta-1)a_\beta,
\quad\text{and}
\\
\coeff{e_{(\beta+2+d)(\beta+1-d)}}{X_{K_\lambda}}/\beta!
&=\mathcal L_d(\beta D_{2d+1})
=\beta(2d+1)a_d.
\qedhere
\end{align*}
\end{proof}

We then compute the~$e$-coefficients for length-three partitions with a
part~$1$.

\begin{lemma}
\label{lem:coeff.len3}
For $\beta\ge 1$ and $0\le d\le \beta$,
\[
\coeff{e_{(\beta+d+1)(\beta+1-d)1}}{X_{K_\lambda}}
=\begin{dcases*}
\beta!,& if $d=0$,
\\
0,& if $d=1$,
\\
\beta!\brk1{(4d-3)a_{d-1}+a_{d-2}},& if $2\le d\le \beta$.
\end{dcases*}
\]
\end{lemma}

\begin{proof}
Set $a=\beta+d+1$ and $b=\beta+1-d$,
and define
\[
f(y_1,y_2,y_3)
=m_{ab1}(y_1,y_2,y_3).
\]
Every monomial involves all three roots,
so $f(r,s,0)=0$.
Thus $F=0$,
and \cref{cor:e-coeff.K.special,lem:FG.f} give
\[
\coeff{e_{ab1}}{X_{K_\lambda}}
=\lambda!\coeff{M}{f(y)}
=\beta!\mathcal L_{\beta+1}(G).
\]

For $d=0$,
we have $a=b=\beta+1$.
The only monomial linear in~$w$ is~$r^as^aw$,
and
\[
G
=r^as^a
=B^{\beta+1}.
\]
It follows that $\coeff{e_\mu}{X_{K_\lambda}}/\beta!=\mathcal
L_{\beta+1}(B^{\beta+1})=1$.

For $1\le d\le \beta$,
we have $a>b\ge 1$.
The monomials linear in~$w$ are~$r^as^bw$ and~$r^bs^aw$.
This remains true when $d=\beta$ and $b=1$,
since repeated parts are counted only once.
Hence
\[
G
=r^as^b+r^bs^a
=B^b(r^{a-b}+s^{a-b})
=B^bD_{a-b}
=B^{\beta+1-d}D_{2d}.
\]
It follows that
\[
\coeff{e_{ab1}}{X_{K_\lambda}}/\beta!
=\mathcal L_{\beta+1}(B^{\beta+1-d}D_{2d})
=\mathcal L_d(D_{2d})
=\Delta_d
=\brk1{(4d-3)a_{d-1}+a_{d-2}}\mathbf 1(d\ge 2),
\]
by \cref{lem:functional-evaluations,prop:a:Delta}.
In fact,
the hypotheses of \cref{lem:functional-evaluations} hold since $0\le
\beta+1-d\le \beta+1$ and since~$D_{2d}$ has~$B$-degree~$d$.
\end{proof}

We can now assemble the coefficient families and prove \cref{thm:e.K}.

\begin{proof}
By \cref{lem:restrction.e-support},
a nonzero coefficient is indexed by a partition of length at most~$3$,
and a length-$3$ index has smallest part one.
The coefficient statements in
\cref{lem:coeffs.short,lem:coeff.len3} therefore exhaust
the possible indices and give the expansion in \cref{thm:e.K}.
The numbers~$a_d$ count restricted injections and are positive.
Every displayed coefficient is consequently nonnegative: the $d=1$ length-$3$
coefficient is zero,
and the endpoint factor $2\beta^2+\beta-1$ is positive for~$\beta\ge 1$.
Hence $X_{K_{32^\beta}}$ is~$e$-positive.
\end{proof}

Stanley's sink theorem \citep[Thm.~3.3]{Sta95},
the fixed-sink formula of \citet[Thm.~7.3]{greene-zaslavsky1983},
and the enumeration $d!a_d$ of acyclic orientations of~$K_{2^d}$
due to \citet[Thm.~3.4]{carballosa-reyes-khera2025}
give an alternative proof of the
top-coefficient formula in \cref{lem:coeffs.short}.
Deleting a prescribed sink from the size-three part of~$K_\lambda$
leaves~$K_{2^{\beta+1}}$;
symmetry among its partite sets shows that a fraction~$\beta/(\beta+1)$ of its
acyclic orientations extend to an orientation with that unique sink.
The fixed-sink formula then gives the same count for every prescribed sink,
so Stanley's theorem recovers the top~$e$-coefficient.
Deleting the entire size-three part and restoring all incident edges pointing
into it similarly gives $\beta!a_\beta$ orientations with three sinks,
which Stanley's theorem identifies with the sum of the length-three
$e$-coefficients.
These orientation counts do not determine the individual length-two and
length-three coefficients or establish their nonnegativity.
Thus they do not yield the full expansion in \cref{thm:e.K} or the
resulting~$e$-positivity classification,
for which the coefficient arguments above remain necessary.

\section*{Disclosure of artificial intelligence use}

We used ChatGPT~6 Pro through its web interface and OpenAI Codex.
An earlier version of this paper used ChatGPT~5.6 Pro.

GPT-5.6 Pro first resolved the main obstacles to proving
\cref{cor:es.K}.
A central step was \cref{cor:e-coeff.K.special},
which reduces the coefficient computation to extracting the coefficient of~$M$
from a monomial symmetric polynomial in~$3$ variables.
GPT-6 Astra introduced the first-order coefficient formula and the
linear-functional evaluations in
\cref{lem:FG.f,lem:functional-evaluations},
which substantially simplified the computations in
\cref{lem:coeffs.short,lem:coeff.len3}.
Codex assisted with computational checks,
literature searches,
and the preparation and revision of the \LaTeX{} source.

We guided the AI-assisted derivation of the full~$e$-expansion in
\cref{thm:e.K}.
While computing the coefficients in \cref{lem:coeffs.short},
we recognized the connection with Dickson polynomials and introduced the
numbers~$a_d$ and~$\Delta_d$ to organize the calculations.
We carried out the symbolic simplification in this paper.

We have independently verified every result stated in this paper and take full
responsibility for its presentation,
accuracy,
and content.

\section*{Statements and Declarations}

\noindent\textbf{Funding.}
The work of David G.L. Wang was supported by the General Program of the
National Natural Science Foundation of China (Grant No.~12171034).

\noindent\textbf{Competing interests.}
The authors have no competing interests to declare.

\noindent\textbf{Data availability.}
Data sharing is not applicable to this article as no datasets were generated or
analyzed during the current study.

\bibliographystyle{plainnat}
\bibliography{../csf}

\end{document}